\documentclass[12pt]{amsart}

\usepackage{amssymb}
\usepackage{bm}
\usepackage{graphicx}
\usepackage{psfrag}
\usepackage{color}
\usepackage{url}
\usepackage{enumerate}

\newcommand{\excise}[1]{}

\makeatletter
\newtheorem*{rep@theorem}{\rep@title}
\newcommand{\newreptheorem}[2]{%
\newenvironment{rep#1}[1]{%
 \def\rep@title{#2 \ref{##1}}%
 \begin{rep@theorem}}%
 {\end{rep@theorem}}}
\makeatother

\newtheorem{thm}{Theorem}[section]
\newtheorem{lemma}[thm]{Lemma}

\newtheorem{cor}[thm]{Corollary}
\newtheorem{prop}[thm]{Proposition}

\newtheorem{prob}[thm]{Problem}

\theoremstyle{definition}
\newtheorem{example}[thm]{Example}
\newtheorem{remark}[thm]{Remark}
\newtheorem{defn}[thm]{Definition}

\numberwithin{equation}{section}

\newreptheorem{cor}{Corollary}

\newcommand{\ring}[1]{\ensuremath{\mathbb{#1}}}

\renewcommand\>{\rangle}
\newcommand\<{\langle}

\newcommand\NN{\ring{N}}

\newcommand\RR{\ring{R}}

\newcommand\ZZ{\ring{Z}}

\newcommand\xx{{\mathbf x}}

\renewcommand\aa{{\mathbf a}}
\newcommand\bb{{\mathbf b}}

\DeclareMathOperator\Betti{Betti} 

\begin{document}

\mbox{}
\title{On the set of catenary degrees of finitely generated cancellative commutative monoids\qquad}
\author{Christopher O'Neill}
\address{Mathematics Department\\Texas A\&M University\\College Station, TX 77843}
\email{coneill@math.tamu.edu}
\author{Vadim Ponomarenko}
\address{Mathematics Department\\San Diego State University\\San Diego, CA 92182}
\email{vponomarenko@mail.sdsu.edu}
\author{Reuben Tate}
\address{Mathematics Department\\University of Hawai`i at Hilo\\Hilo, HI 96720}
\email{reubent@hawaii.edu}
\author{Gautam Webb}
\address{Mathematics Department\\Colorado College\\Colorado Springs, CO 80903}
\email{gautam.webb@coloradocollege.edu}

\date{\today}

\begin{abstract}
\hspace{-2.05032pt}
The catenary degree of an element $n$ of a cancellative commutative monoid $S$ is a nonnegative integer measuring the distance between the irreducible factorizations of $n$.  The catenary degree of the monoid $S$, defined as the supremum over all catenary degrees occurring in $S$, has been studied as an invariant of nonunique factorization.  In this paper, we investigate the set $\mathsf C(S)$ of catenary degrees achieved by elements of $S$, focusing on the case where $S$ in finitely generated (where $\mathsf C(S)$ is known to be finite).  Answering an open question posed by Garc\'ia-S\'anchez, we provide a method to compute the smallest nonzero element of $\mathsf C(S)$ that parallels a well-known method of computing the maximum value.  We also give several examples demonstrating certain extremal behavior for $\mathsf C(S)$, and present some open questions for further study.  
\end{abstract}
\maketitle


\vspace{-1cm}

\section{Introduction}
\label{s:intro}

Nonunique factorization theory aims to classify and quantify the failure of elements of cancellative commutative monoids to factor uniquely into irreducibles \cite{nonuniq}.  Factorization invariants are arithmetic quantities that measure the failure of a monoid's elements to admit unique factorizations.  There are many standard invariants used in the literature to compare factorization behavior between monoids, such as the delta set \cite{delta}, elasticity \cite{elasticity}, and $\omega$-primality \cite{prime,quasi}.  

Most factorization invariants assign a value to each monoid element determined by its factorization structure.  In examining the behavior of an invariant throughout the whole monoid, one often considers quantities such as the supremum of all values attained at its elements, or the set of all such values.  For instance, the elasticity of a monoid element is defined as the ratio of its largest and smallest factorization lengths, and the elasticity of the monoid is simply the supremum of the elasticities of its elements.  

This paper concerns the catenary degree (Definition~\ref{d:catenarydegree}), a factorization invariant that has been the subject of much recent work \cite{semitheor,catenarytamenumerical,omidali}.  The catenary degree $\mathsf c(n)$ of a monoid element $n \in S$ is a nonnegative integer derived from combinatorial properties of the set of factorizations of $n$.   Although much of the literature on the catenary degree focuses on the maximum catenary degree attained within $S$, some recent papers \cite{catenaryperiodic,catenaryarith} examine the catenary degees of individual monoid elements.  In this paper, we investigate the set $\mathsf C(S)$ of catenary degrees occurring within $S$ as a factorization invariant, focusing on the setting where $S$ is finitely generated.  

The catenary degree $\mathsf c(n)$ of a monoid element $n \in S$ is closely related to the delta set $\Delta(n)$ (Definition~\ref{d:factorizations}), although neither completely determines the other.  The factorization invariants $\Delta(S) = \bigcup_{n \in S} \Delta(n)$ and $\mathsf c(S) = \max_{n \in S} \mathsf c(n)$, both of which are standard in the literature, describe the factorization structure of $S$ at drastically different levels of detail.  Indeed, one maintains a comprehensive set of values, while the other records only a single value.  The set $\mathsf C(S)$ of catenary degrees of $S$ more closely resembles the delta set $\Delta(S)$, as it provides a more comparably detailed description of the factorization structure of $S$.  The goal of this paper is to provide an initial examination of the set $\mathsf C(S)$, as well as to motivate further study.  

It is known from \cite{catenarytamefingen} that $\max \mathsf C(S)$, the maximum catenary degree attained within a given monoid $S$, always occurs at least once within a certain finite class of elements of $S$ called Betti elements (Definition~\ref{d:bettielement}).  This result was later extended to monoids satisfying a certain weakened Noetherian condition \cite{catenarymaxgeneral}.  The second author of \cite{catenarytamefingen} conjectured that when $S$ is finitely generated, the minimum nonzero value of $\mathsf C(S)$ also occurs at a Betti element of $S$.  

As the main result of this paper, we prove the aformetioned conjecture (Theorem~\ref{c:mincatdegree}).  This yields a computable bound on the values occurring in $\mathsf C(S)$, as well as a characterization of those finitely gererated monoids $S$ for which $|\mathsf C(S)|$ is minimal.  We also give in Example~\ref{e:catenaryset} a semigroup whose set of catenary degrees has a nonzero element that does not occur at a Betti element, demonstrating that this result need not extend to the entire set $\mathsf C(S)$.  

In Section~\ref{s:extremecatset}, we give several examples that demonstrate certain extremal behavior of $\mathsf C(S)$.  In particular, we show that if a monoid $S$ has at least $3$ minimal generators, then $|\mathsf C(S)|$ cannot be bounded by the number of generators of $S$ (Theorem~\ref{t:largecatset}).  

\subsection*{Acknowledgements}
The authors would like to thank Scott Chapman and Pedro Garc\'ia-S\'anchez for numerous helpful conversations, and Alfred Geroldinger for his assistance with Theorem~\ref{t:fullset}.

\section{Background}
\label{s:background}

Unless otherwise stated, $S$ denotes a finitely generated cancellative commutative monoid, written additively.  

\begin{remark}\label{r:mingenset}
By passing from $S$ to the quotient by its unit group (if necessary), we will assume that $S$ is \emph{reduced}, that is, $S$ has no nonzero units.  In this case, $S$ has a unique generating set that is minimal with respect to containment, namely the set of irreducible elements of $S$.  When we write $S = \<n_1, \ldots n_k\>$, it is assumed that the elements $n_1, \ldots, n_k$ are precisely the irreducible elements of $S$.  
\end{remark}

\begin{defn}\label{d:numerical}
An additive submonoid $S \subset \NN$ is a \emph{numerical monoid} if $|\NN \setminus S| < \infty$.  If we write $S = \<n_1, \ldots, n_k\>$, it is assumed that the integers $n_1 < \cdots < n_k$ comprise the unique minimal generating set of $S$.  
\end{defn}

\begin{defn}\label{d:factorizations}
Fix $n \in S$.  A \emph{factorization} of $n$ is an expression $n = u_1 + \cdots + u_r$ of $n$ as a sum of irreducible elements $u_1, \ldots, u_r$ of $S$.  If $S = \<n_1, \ldots, n_k\>$, we often write factorizations of $n$ in the form $n = a_1n_1 + \cdots + a_kn_k$ (see Remark~\ref{r:mingenset}).  Write 
$$\mathsf Z_S(n) = \{(a_1, \ldots, a_k) : n = a_1n_1 + \cdots + a_kn_k\} \subset \NN^k$$
for the \emph{set of factorizations} of $n \in S$.  Given $\aa \in \mathsf Z_S(n)$, we denote by $|\aa|$ the number of irreducibles in the factorization $\aa$, that is, $|\aa| = a_1 + \cdots + a_k$.  The set of factorization lengths of $n$, denoted $\mathsf L_S(n) = \{|\aa| : \aa \in \mathsf Z_S(n)\}$, is called the \emph{length set} of $n$.  Writing $\mathsf L(S) = \{\ell_1 < \cdots < \ell_r\}$, the \emph{delta set} of $n$ is the set $\Delta(n) = \{\ell_i - \ell_{i-1} : 2 \le i \le r\}$ of successive differences of factorization lengths of $n$, and $\Delta(S) = \bigcup_{n \in S} \Delta(n)$.  When there is no ambiguity, we often omit the subscripts and write $\mathsf Z(n)$, $\mathsf L(n)$, and $\Delta(n)$.  
\end{defn}

\begin{defn}\label{d:catenarydegree}
Fix an element $n \in S$.  For $\aa, \bb \in \mathsf Z(n)$, the \emph{greatest common divisor of $\aa$ and $\bb$} is given by 
$$\gcd(\aa,\bb) = (\min(a_1,b_1), \ldots, \min(a_r,b_r)) \in \NN^r,$$
and the \emph{distance between $\aa$ and $\bb$} (or the \emph{weight of $(\aa,\bb)$}) is given by 
$$d(\aa,\bb) = \max(|\aa - \gcd(\aa,\bb)|,|\bb - \gcd(\aa,\bb)|).$$
Given $\aa, \bb \in \mathsf Z(n)$ and $N \ge 1$, an \emph{$N$-chain from $\aa$ to $\bb$} is a sequence $\aa_1, \ldots, \aa_k \in \mathsf Z(n)$ of factorizations of $n$ such that (i) $\aa_1 = \aa$, (ii) $\aa_k = \bb$, and (iii) $d(\aa_{i-1},\aa_i) \le N$ for all $i \le k$.  The \emph{catenary degree of $n$}, denoted $\mathsf c(n)$, is the smallest non-negative integer $N$ such that there exists an $N$-chain between any two factorizations of $n$.  The \emph{set of catenary degrees of $S$} is the set $\mathsf C(S) = \{\mathsf c(m) : m \in S\}$.  
\end{defn}

\begin{example}\label{e:catenarydegree}
Consider the numerical monoid $S = \<11,36,39\>$.  The left-hand picture in Figure~\ref{fig:catenarydegree} depicts the factorizations of $450 \in S$ along with all pairwise distances.  There exists a $16$-chain between any two factorizations of $450$; one such $16$-chain between $(6,2,8)$ and $(24,3,2)$ is depicted with bold red edges.  Since every $16$-chain between these factorizations contains the edge labeled 16 at the bottom, we have $\mathsf c(450) = 16$.  

This can also be computed in a different way.  In the right-hand picture of Figure~\ref{fig:catenarydegree}, only distances of at most $16$ are depicted, and the resulting graph is connected.  Removing the edge labeled $16$ yields a disconnected graph, so we again conclude $\mathsf c(450) = 16$.  
\end{example}

\begin{figure}
\includegraphics[width=2.5in]{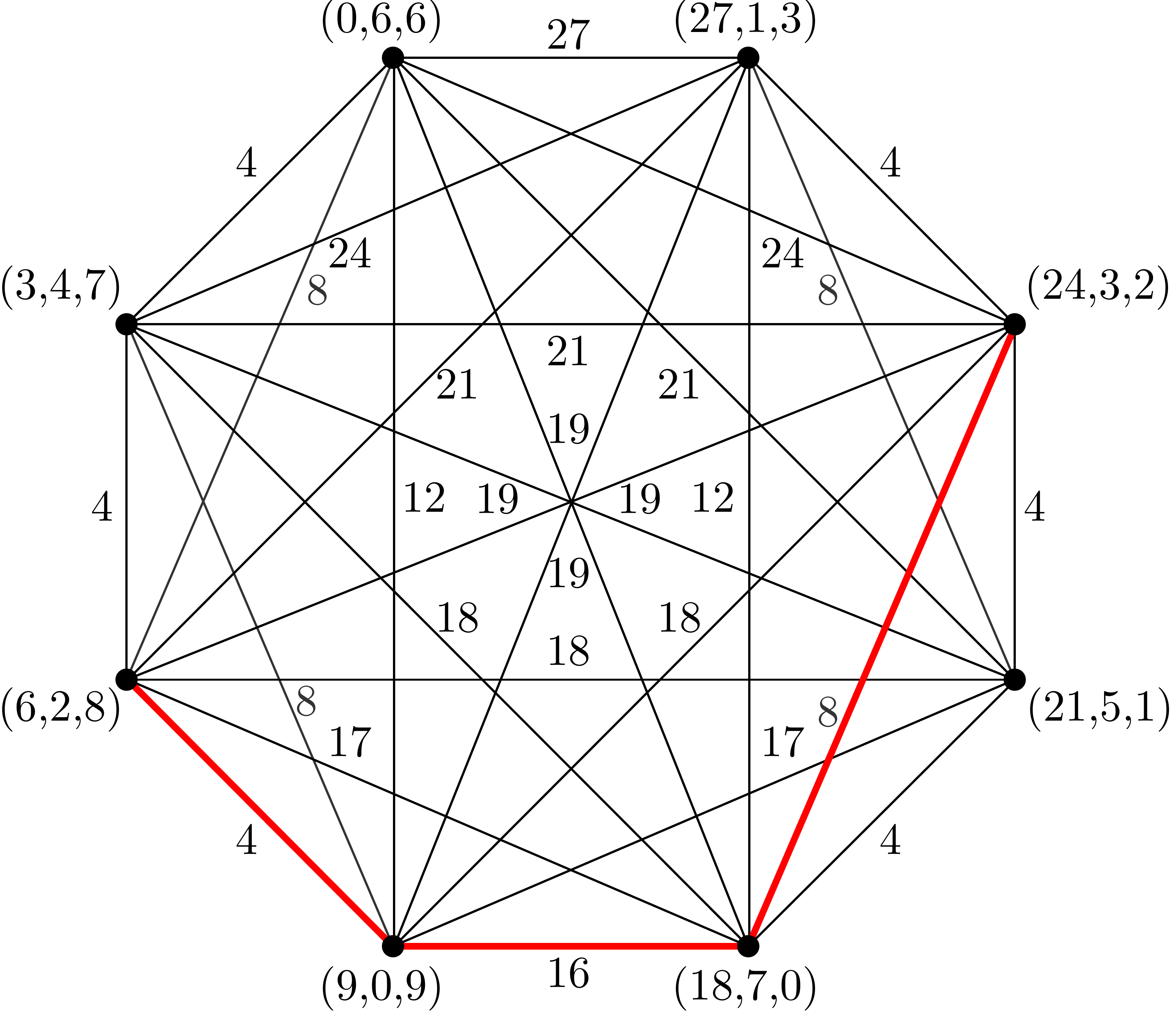}
\hspace{0.5in}
\includegraphics[width=2.5in]{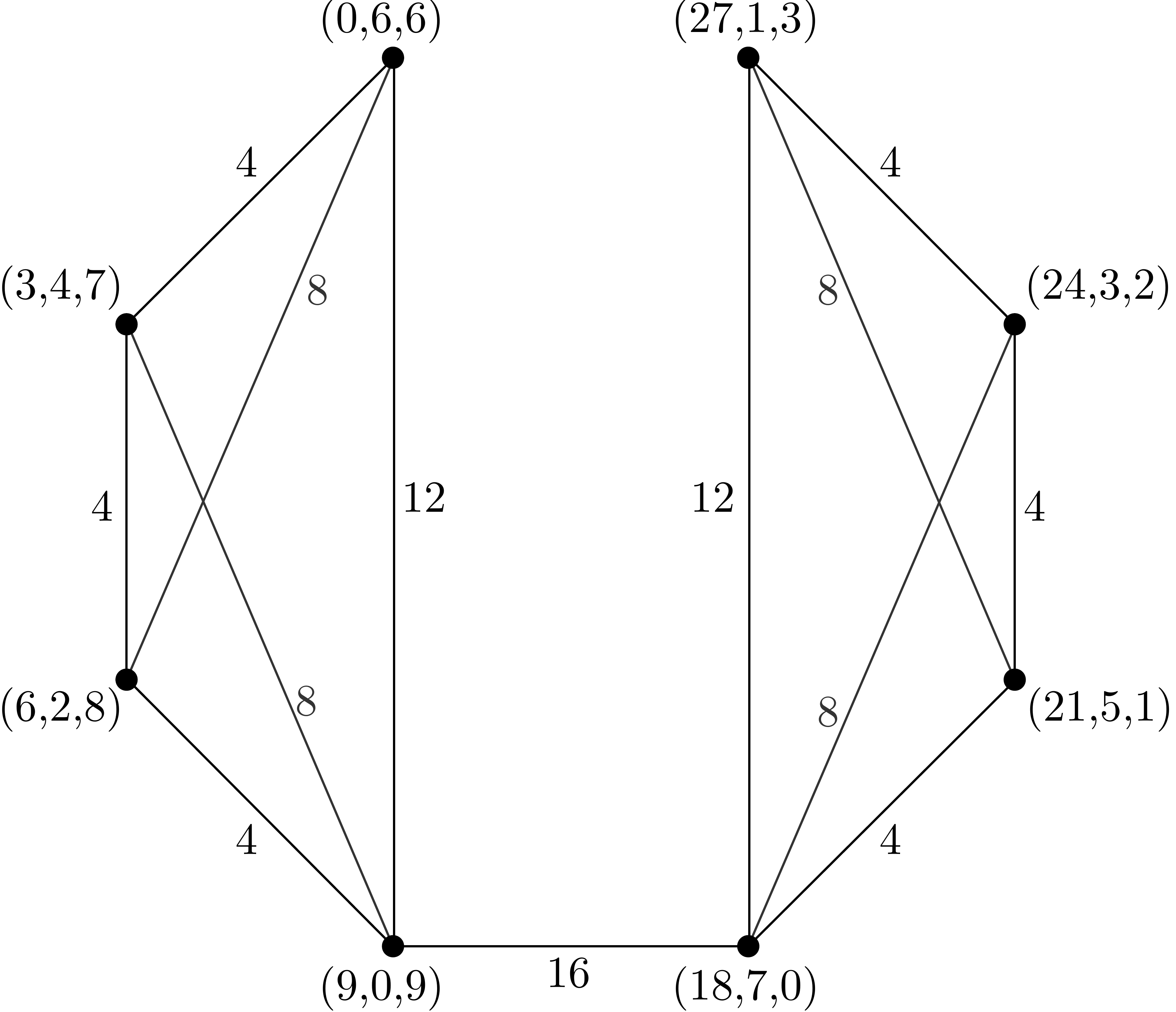}
\medskip
\caption{Computing the catenary degree of $450 \in S = \<11,36,39\>$, as in Example~\ref{e:catenarydegree}.  Each vertex is labeled with an element of $\mathsf Z(450)$, and each edge is labeled with the distance between the factorizations at either end.  The diagram on the left depicts all edges, and the diagram on the right includes only those edges labeled at most $\mathsf c(450) = 16$.  Both graphics were created using the computer algebra system \texttt{SAGE} \cite{sage}.}  
\label{fig:catenarydegree}
\end{figure}

Factorizations of certain monoid elements, called Betti elements (Definition~\ref{d:bettielement}), contain much of the structural information used to construct chains of factorizations, and thus are closely related to the catenary degree.  For instance, it is known that the maximum catenary degree occurring in a monoid $S$ is guaranteed to occur at a Betti element of $S$ (Theorem~\ref{t:maxcatdegree}).  Continuing in this vein, we will show in Corollary~\ref{c:minmaxcatdegree} that the minimum nonzero catenary degree of $S$ also occurs at a Betti element of $S$.  

\begin{defn}\label{d:bettielement}
Fix a finitely generated monoid $S$.  For each nonzero $n \in S$, consider the graph $\nabla_n$ with vertex set $\mathsf Z(n)$ in which two vertices $\aa, \bb \in \mathsf Z(n)$ share an edge if $\gcd(\aa,\bb) \ne 0$.  If $\nabla_n$ is not connected, then $n$ is called a \emph{Betti element} of $S$.  We write 
$$\Betti(S) = \{b \in S : \nabla_b \text{ is disconnected}\}$$
for the set of Betti elements of $S$.  
\end{defn}

\begin{thm}[{\cite[Theorem~3.1]{catenarytamefingen}}]\label{t:maxcatdegree}
For any finitely generated monoid $S$, 
$$\max \mathsf C(S) = \max\{\mathsf c(b) : b \in \Betti(S)\}.$$
\end{thm}

\section{The minimum nonzero catenary degree}
\label{s:mincatdegree}

By Theorem~\ref{t:maxcatdegree}, in order to compute the maximum catenary degree achieved in a monoid $S$, it suffices to compute the catenary degree of each of its Betti elements.  The same cannot be said for all of the values in $\mathsf C(S)$; see Example~\ref{e:catenaryset}.  The main result of this section is Theorem~\ref{t:mincatdegree}, which implies that the minimum nonzero catenary degree achieved in $S$ also occurs at a Betti element of $S$.  

\begin{example}\label{e:catenaryset}
Let $S = \<11,25,29\> \subset \NN$.  The catenary degrees of $S$ are plotted in Figure~\ref{fig:catenaryset}.  The only Betti elements of $S$ are 58, 150, and 154, which have catenary degrees 4, 12, and 14, respectively.  However, $\mathsf c(175) = 11$ is distinct from each of these values.  However, by Corollary~\ref{c:minmaxcatdegree}, every element of $S$ with at least two distinct factorizations has catenary degree at least 4 and at most 14.  
\end{example}

\begin{figure}
\includegraphics[width=5in]{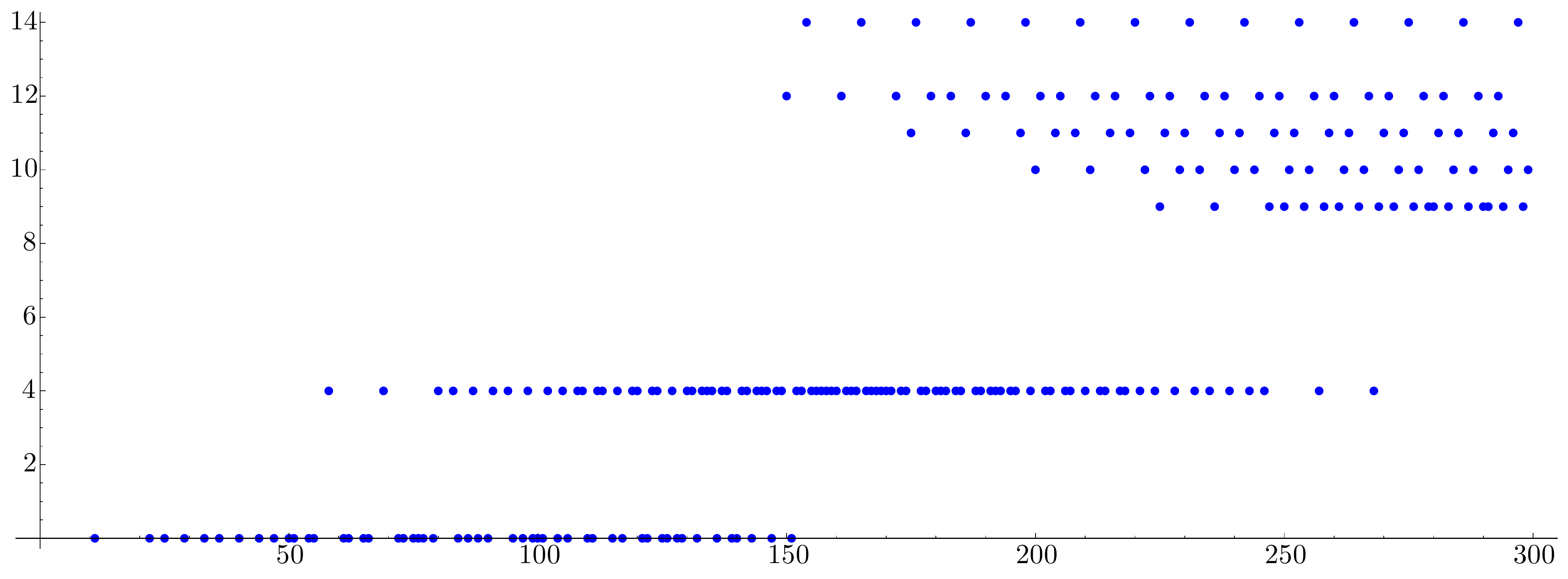}
\medskip
\caption{A \texttt{SAGE} plot \cite{sage} showing the catenary degrees for the numerical monoid $S = \<11,25,29\>$ discussed in Example~\ref{e:catenaryset}.}
\label{fig:catenaryset}
\end{figure}


\begin{remark}\label{r:mincatdegree}
Fix $n \in S$, and let $b$ denote the minimum catenary degree among Betti elements dividing $n$.  Theorem~\ref{t:mincatdegree} is proved using a combinatorial argument that edges with weight strictly less than $b$ are not sufficient to connect all factorizations of $n$.  The key idea in the proof is Proposition~\ref{p:mincatdegree}, which implies that any factorization with an edge of weight strictly less than $b$ cannot have maximal length in $\mathsf Z(n)$.  
\end{remark}

\begin{lemma}\label{l:mincatdegree}
Suppose $S = \<n_1, \ldots, n_k\>$.  Fix $n \in S$ with $|\mathsf Z(n)| \ge 2$, let $B$ be the set of Betti elements of $S$ that divide $n$, and let 
$$b = \min\{\mathsf c(m) : m \in B\}.$$
For each $\aa = (a_1, \ldots, a_k) \in Z(n)$, there exists $\aa' \in Z(n)$ such that $d(\aa,\aa') \ge b$.  
\end{lemma}

\begin{proof}
Let $X = \{(x_1, \ldots, x_k) : 0 \le x_i \le a_i \text{ for } 1 \le i \le k\}$ and let 
$$F = \{x \in X: |\mathsf Z(x_1n_1 + \cdots + x_kn_k)| \ge 2\} \subset X.$$
Note that $F$ forms a finite nonempty partially ordered set with unique maximal element $\aa$.  Choose a minimal element $\bb = (b_1, \ldots, b_k) \in F$, and let $m = b_1n_1 + \ldots + b_kn_k$.  

Minimality of $\bb$ implies that $|\mathsf Z(m - n_i)| = 1$ for each positive $b_i$, so any factorization $\bb' \in \mathsf Z(m)$ with $\bb' \ne \bb$ satisfies $\gcd(a_i',b_i') = 0$.  In particular, $m \in B$.   Fixing $\bb' \in \mathsf Z(m)$ distinct from $\bb$, and choosing $\aa' = \bb' + \aa - \bb \in \mathsf Z(n)$, we have 
$$d(\aa, \aa') = d(\bb + \aa - \bb, \bb' + \aa - \bb) = d(\bb, \bb') \ge \mathsf c(m) \ge b,$$
as desired.  
\end{proof}

\begin{prop}\label{p:mincatdegree}
Suppose $S = \<n_1, \ldots, n_k\>$.  Fix $n \in S$ with $|\mathsf Z(n)| \ge 2$, let $B$ be the set of Betti elements of $S$ that divide $n$, and let 
$$b = \min\{\mathsf c(m) : m \in B\}.$$
Given distinct $\aa, \bb \in \mathsf Z(n)$ with $d(\aa, \bb) < b$, there exists $\xx \in \mathsf Z(n)$ such that 
$$\max\{|\aa|, |\bb|\} < |\xx|.$$
\end{prop}

\begin{proof}
First, suppose $\gcd(\aa, \bb) = 0$, so that $d(\aa, \bb) = \max\{|\aa|, |\bb|\}$.  By Lemma~\ref{l:mincatdegree}, there exists $\xx \in Z(n)$ such that $d(\aa, \xx) \ge b$.  
The strict inequality
$$|\aa - \gcd(\aa, \xx)| \le |\aa| < b \le d(\aa, \xx) = \max\{|\aa - \gcd(\aa, \xx)|, |\xx - \gcd(\aa, \xx)|\}$$
implies $d(\aa, \xx) = |\xx - \gcd(\aa, \xx)|$.  This means 
$$\max\{|\aa|, |\bb|\} = d(\aa, \bb) < b \le d(\aa, \xx) = |\xx - \gcd(\aa, \xx)| \le |\xx|,$$
which proves the claim in this case.  

Now, suppose $\gcd(\aa, \bb) \ne 0$.  Let $\aa' = \aa - \gcd(\aa, \bb)$ and $\bb' = \bb - \gcd(\aa, \bb)$, and fix $n' \in S$ such that $\aa', \bb' \in \mathsf Z(n')$.  Since any Betti element dividing $n'$ also divides $n$, the above argument ensures the existence of $\xx' \in \mathsf Z(n')$ such that 
$$\max\{|\aa|, |\bb|\} = \max\{|\aa'|, |\bb'|\} + |\gcd(\aa, \bb)| < |\xx'| + |\gcd(\aa, \bb)|.$$
Choosing $\xx = \xx' + \gcd(\aa,\bb)$ completes the proof.  
\end{proof}

\begin{thm}\label{t:mincatdegree}
Suppose $S = \<n_1, \ldots, n_k\>$.  Fix $n \in S$ with $|\mathsf Z(n)| \ge 2$, and let $B$ denote the set of Betti elements of $S$ that divide $n$.  Then 
$$\mathsf c(n) \ge \min\{\mathsf c(m) : m \in B\}.$$
\end{thm}

\begin{proof}
Let $b = \min\{\mathsf c(m) : m \in B\}$, and let 
$$V = \{\aa \in \mathsf Z(n) : d(\aa, \bb) < b \text{ for some } \bb \in \mathsf Z(n)\} \subset \mathsf Z(n).$$
If $V = \emptyset$, then $d(\aa, \bb) \ge b$ for all $\aa, \bb \in \mathsf Z(n)$, and it follows that $\mathsf c(n) \ge b$.  Otherwise, choose $\aa \in V$ such that $|\aa|$ is maximal among elements of $V$.  Since $\aa \in V$, there exists $\bb \in \mathsf Z(n)$ such that $d(\aa, \bb) < b$.  By Proposition~\ref{p:mincatdegree}, there exists $\xx \in \mathsf Z(n)$ such that $\max\{|\aa|, |\bb|\} < |\xx|$.  Since $|\aa| \le \max\{|\aa|, |\bb|\} < |\xx|$, maximality of $|\aa|$ ensures that $\xx \not\in V$.  Consequently, $d(\xx,\xx') \ge b$ for all $\xx' \in Z(n)$ with $\xx' \ne \xx$, so $\mathsf c(n) \ge b$.  
\end{proof}

We conclude this section with several immediate consequences Theorem~\ref{t:mincatdegree}.  The first is Corollary~\ref{c:mincatdegree}, in the spirit of Theorem~\ref{t:maxcatdegree}.  

\begin{cor}\label{c:mincatdegree}
If $n \in S$ satisfies $\mathsf c(n) > 0$, then 
$$\mathsf c(n) \ge \min\{\mathsf c(m) : m \in \Betti(S)\}.$$
In particular, $\min(\mathsf C(S) \setminus \{0\})$ is the catenary degree of some Betti element of $S$.  
\end{cor}

\begin{remark}\label{r:dividingbetti}
The proof of Theorem~\ref{t:maxcatdegree} given in \cite{catenarytamefingen} can be easily extended to show that the catenary degree of any monoid element is bounded above by the catenary degrees of the Betti elements dividing it.  We record this in Corollary~\ref{c:minmaxcatdegree}.  
\end{remark}

\begin{cor}\label{c:minmaxcatdegree}
Fix $n \in S$ with $\mathsf c(n) > 0$, and let $B$ denote the set of Betti elements of $S$ that divide $n$.  Then 
$$\min\{\mathsf c(m) : m \in B\} \le \mathsf c(n) \le \max\{\mathsf c(m) : m \in B\}.$$
\end{cor}

Lastly, Corollary~\ref{c:singlecatdegree} classifies those monoids $S$ for which $|\mathsf C(S)|$ is minimal, and generalizes \cite[Theorem~19]{uniquebetti}.  

\begin{cor}\label{c:singlecatdegree} 
$\mathsf C(S)=\{0, c\}$ if and only if $\mathsf c(m) = c$ for all $m \in \Betti(S)$.  
\end{cor}

\begin{remark}\label{r:mindelta}
The set $\mathsf C(S)$ of catenary degrees occurring in a monoid $S$ shares many similarities to the delta set $\Delta(S)$.  In fact, the maximum element of $\Delta(S)$ is known to lie in the delta set of a Betti element \cite{semitheor}.  In contrast, this need not hold for the minimum element of $\Delta(S)$.  For example, the numerical monoid $S = \<30,52,55\>$ has delta set $\Delta(S) = \{1,2,3,5\}$, but its only Betti elements are 260 and 330, and their delta sets are given by $\Delta(260) = \{2\}$ and $\Delta(330) = \{5\}$.  
\end{remark}

In general, it is not easy to prove that a given value $c$ does not equal the catenary degree of any elements of a given monoid $S$ (the same difficulty arises when computing the delta set of a monoid; see Remark~\ref{r:mindelta}).  Computer software can be used to compute the catenary degree of individual elements of $S$ (for instance, the \texttt{GAP} package \texttt{numericalsgps} \cite{numericalsgpsgap} can do this).  However, computing $\mathsf C(S)$ via exhaustive search is not possible, and it can be difficult to determine when the whole set $\mathsf C(S)$ has been computed.  An answer to Problem~\ref{prob:gammi} would allow for a more effective use of computer software packages in studying $\mathsf C(S)$.  

\begin{prob}\label{prob:gammi}
Given a monoid $S$, determine a (computable) finite class of elements of $S$ on which every catenary degree in $\mathsf C(S)$ occurs.  
\end{prob}

\section{Some extremal examples of $\mathsf C(S)$}
\label{s:extremecatset}

The delta set realization problem \cite{deltarealiz} asks which finite sets $D \subset \NN$ satisfy $\Delta(S) = D$ for some monoid $S$.  The results of this section pertain to the analogous question for sets of catenary degrees, which we record this here as Problem~\ref{pr:realization}.  Note that for any $S$, Definition~\ref{d:catenarydegree} implies $1 \notin C(S)$ and every atom of $S$ has catenary degree 0.  

\begin{prob}\label{pr:realization}
Fix a finite set $C \subset \NN$ such that $C \cap \{0,1\} = \{0\}$.  Does there exist a finitely generated monoid $S$ with $\mathsf C(S) = C$?  
\end{prob}

This section aims to provide an initial investigation for Problem~\ref{pr:realization} by demonstrating some extremal properties sets of catenary degrees can achieve.  In particular, 

\begin{enumerate}[(i)]
\item 
we exhibit finitely generated monoids $S$ achieving sets of catenary degrees with extremal cardonality and density (as a subset of $\{0\} \cup [2, \mathsf c(S)] $), and 

\item 
we establish the independence of these properties from the number of atoms (with one exception; see Remark~\ref{r:fullset}).  

\end{enumerate}

We begin with Remark~\ref{r:singlecatdegree}, which identifies a family of monoids achieving a single nonzero catenary degree, followed by Theorem~\ref{t:arithmetical}, which identifies a family of monoids whose only nonzero catenary degrees are 2 and the maximum.  Both families consist of numerical monoids, and within each family, the number of atoms can be chosen arbitrarily large.  


\begin{remark}\label{r:singlecatdegree}
Corollary~\ref{c:singlecatdegree} classifies monoids $S$ with exactly one nonzero element in $\mathsf C(S)$.  Such monoids can also have arbitrarily large minimal generating sets.  In particular, if $p_1 < \cdots < p_k$ are $k$ distinct primes, then the numerical monoid 
$$S = \<(p_1 \cdots p_k)/p_k, \ldots, (p_1 \cdots p_k)/p_1\>$$
has a single Betti element, so $\mathsf C(S) = \{0,p_k\}$ by Corollary~\ref{c:singlecatdegree}.  See \cite{uniquebetti} for more detail on monoids with a unique Betti element.  
\end{remark}



\begin{thm}\label{t:arithmetical}
Fix $c \ge 3$ and $k \ge 3$.  Let $S = \<k, k + (c - 2), \ldots, k + (k - 1)(c - 2)\>$.  Then $\mathsf C(S) = \{0, 2, c\}$.  
\end{thm}

\begin{proof}
Since $S$ is generated by an arithmetic sequence, apply \cite[Theorem~3.1]{catenaryarith}.  
\end{proof}

In the remainder of this section, we exhibit two infinite families of finitely generated monoids whose sets of catenary degrees have arbitrarily large cardonality.  First, Theorem~\ref{t:largecatset} defines an infinite family of 3-generated numerical monoids with this property (Example~\ref{e:catenaryset} depicts a numerical monoid from this family).  Second, Theorem~\ref{t:fullset} exhibits an infinite family of block monoids (Definition~\ref{d:blockmonoid}) whose sets of catenary degrees have no missing values between 0 and $\mathsf c(S)$ (other than 1).  

Before proving Theorem~\ref{t:largecatset}, we recall Lemmas~\ref{l:emb2memcrit} and~\ref{l:cini}, each found in \cite{numerical}.  

\begin{lemma}\label{l:emb2memcrit}
If $S = \<n_1, n_2\>$ is a numerical monoid and $n \in \ZZ$, then $n \in S$ if and only if $n_1n_2 - n_1 - n_2 - n \notin S$.  
\end{lemma}

\begin{lemma}\label{l:cini}
Let $S = \<n_1, n_2, n_3\> \subset \NN$ be a numerical monoid.  Each element of $\Betti(S)$ can be written in the form 
$$c_in_i = r_{ij}n_j + r_{ik}n_k,$$
where $\{i, j, k\} = \{1, 2, 3\}$ and $c_i = \min\{c > 0 : cn_i \in \<n_j, n_k\>\}$.  
\end{lemma}

\begin{thm}\label{t:largecatset}
Fix $k \ge 3$, and let $S = \<n_1, n_2, n_3\> = \<2k + 1, 6k - 5, 6k - 1\>$.  
\begin{enumerate}
\item[(i)] The Betti elements of $S$ are $u = (3k - 1)n_1$, $v = (k + 1)n_2$ and $w = 2n_3$.  
\item[(ii)] We have $\{4, 3k - 1\}, \{2k - 1, \ldots, 3k - 3\} \subset \mathsf C(S)$.  
\end{enumerate}
\end{thm}

\begin{proof}
Notice the generators of $S$ are all pairwise coprime.  Fix $a < 3k - 1$, and write $a = 3b + c$ for $0 \le c < 3$.  We have 
$$\begin{array}{rcl}
n_2n_3 - n_2 - n_3 - an_1
&=& (6k - 5)(6k - 1) - (6k - 5) - (6k - 1) - a(2k + 1) \\
&=& 36k^2 - (2a + 48)k + (11 - a) \\
&=& (2ck + b - 1)n_2 + (6k - 2ck - 2b + c - 6)n_3,
\end{array}$$
so by Lemma~\ref{l:emb2memcrit}, $an_1 \notin \<n_2, n_3\>$.  This means $u = (3k - 1)n_1 = kn_2 + n_3$ is a Betti element of $S$ by Lemma~\ref{l:cini}.  Similarly, for each $a < k + 1$, we have
$$\begin{array}{rcl}
n_1n_3 - n_1 - n_3 - an_2
&=& 12k^2 - (6a + 4)k + (5a - 1) \\
&=& (3a - 1)n_1 + (2k - 2a)n_3,
\end{array}$$
so $v = (k + 1)n_2 = (3k - 4)n_1 + n_3$ is also a Betti element of $S$ by Lemmas~\ref{l:emb2memcrit} and~\ref{l:cini}.  Applying Lemmas~\ref{l:emb2memcrit} and~\ref{l:cini} once more, we conclude from 
$$n_1n_2 - n_1 - n_2 - n_3 = 12k^2 - 18k = (3k - 5)n_1 + (k - 1)n_2,$$
that $w = 2n_3 = 3n_1 + n_2$ is the last Betti element of $S$.  This proves~(i).  

It is easy to check that $u$, $v$, and $w$ each have only 2 distinct factorizations, and the first containment of~(ii) follows from computing $\mathsf c(u) = 3k - 1$ and $\mathsf c(w) = 4$.  For $0 \le j \le k - 2$, let $s_j = 6k^2 + (6j + 1)k - 5j - 5$.  We claim each $s_j$ has exactly $j + 2$ distinct factorizations: $s_j = (k + 1 + j)n_2$, which we shall denote by $\aa_0 \in \mathsf Z(s_j)$ and 
$$s_j = (3k - 1 - 3i)n_1 + (j + 1 - i)n_2 + (2i - 1)n_3$$
for $1 \le i \le j + 1$, which we shall denote by $\aa_i \in \mathsf Z(s_j)$.  Indeed, this has already been shown for $s_0 = v$ above, and induction on $j$ implies each $s_j = s_{j-1} + n_2$ has exactly $j + 1$ factorizations with at least one copy of $n_2$.  Since $3k - 4 - 3j < n_3$ and $2j + 1 < n_1$, the only factorization of $s_j$ in $\<n_1, n_3\>$ is $\aa_0$, from which the claim follows.  

Lastly, if $\mathsf c(s_j) = N$, there exists an $N$-chain eminating from $\aa_0$, so 
$$N \ge \min\{d(\aa_0, \aa_i) : 1 \le i \le j + 1\} = \min\{3k - i - 2 : 1 \le i \le j + 1\} = 3k - 3 - j.$$
Since $d(\aa_i, \aa_{i+1}) = 4$ for $1 \le i \le j$, we have $\mathsf c(s_j) = 3k - 3 - j$.  
\end{proof}

\begin{remark}\label{r:2gen}
Theorem~\ref{t:largecatset} cannot be improved to only require 2 minimal generators.  Indeed, suppose $S \subset \NN^k$ has two atoms.  If $k \ge 2$ and the generators of $S$ are linearly independent, then $S$ is factorial, so $\mathsf C(S) = \{0\}$.  Otherwise, $S$ is isomorphic to a numerical monoid $\<n_1, n_2\> \subset \NN$, and \cite[Remark~2.2]{catenaryarith} implies $\mathsf C(S) = \{0, n_2\}$.  
\end{remark}

The final result of this section concerns monoids of zero-sum sequences over finite groups (Definition~\ref{d:blockmonoid}).  Here, we only introduce what is needed to prove Theorem~\ref{t:fullset}; the unfamiliar reader should consult~\cite{nonuniq} for a more thorough introduction.  

\begin{defn}\label{d:blockmonoid}
Fix a finite Abelian group $G$ with $|G| \ge 3$, written additively, and let $\mathcal F(G)$ denote the (multiplicatively written) free abelian monoid with basis $G$.  An element $A = g_1 \cdot \ldots \cdot g_{\ell} \in \mathcal F(G)$ (called a \emph{sequence} over $G$) is said to be \emph{zero-sum} if $g_1 + \ldots + g_\ell = 0$ in $G$.  The set $\mathcal B(G) \subset \mathcal F(G)$ of zero-sum sequences over $G$ is a submonoid of $\mathcal F(G)$, called the \emph{block monoid of $G$}.   
\end{defn}

\begin{remark}\label{r:blockmonoid}
By \cite[Proposition 2.5.6]{nonuniq}, the block monoid $\mathcal B(G)$ of a finite group $G$ is a Krull monoid with class group isomorphic to $G$ and every class contains a prime divisor.  The catenary degree of block monoids was recently studied in the context of Krull monoids in \cite{krullcata,krullcatb}.  
\end{remark}

\begin{thm}\label{t:fullset}
Let $S$ be the block monoid of a cyclic group $G$ of order $|G| = n \ge 4$, and fix an element $g \in G$ with order $|g| = n$.  Then $\mathsf C(S) = \{0, 2, 3, \ldots, n\}$.
\end{thm}

\begin{proof}
By \cite[Theorem 6.4.7]{nonuniq}, we have $\mathsf c(S) = n$ and hence $\mathsf C(S) \subset \{0, 2,3, \ldots, n\}$.  Hence, it remains to show that the interval $[2, n] \subset \mathsf C (S)$.  First, consider the element 
$$A = (2g)^2 g^{2n-4} \in S.$$
The only minimal (that is, irreducible) zero-sum sequences dividing $A$ are given by $U = g^n$, $V = (2g)g^{n-2}$, and $W = (2g)^2g^{n-4}$.  This yields $\mathsf Z(A) = \{UW, V^2\}$, from which we conclude that $\mathsf c(A) = 2$.  

Now, fix $j \in [3,n]$ and let $h = (j-1)g$.  Consider the zero-sum sequence 
$$A = (-g)^{j-1}g^n h \in S.$$
This time, there are precisely four minimal zero-sum sequences dividing $A$, namely $U = g^n$, $V = g(-g)$, $W = (-g)^{j-1}h$, and $X = g^{n-j+1}h$.  From this, we conclude that $\mathsf Z(A) = \{UW, V^{j-1} X\}$, which means in particular that $\mathsf c(A) = j$.
\end{proof}

\begin{remark}\label{r:fullset}
In Theorem~\ref{t:fullset}, the block monoid with set of catenary degrees $\{0, 2, 3, \ldots, n\}$ has a number of atoms that is exponential in $n$ \cite{vadimsumseq}.  While it is unclear whether this result can be strengthened to use monoids with a bounded number of atoms, such an improvement would require monoids with at least 4 atoms.  Indeed, suppose $S \subset \NN^k$ has only 3 atoms, and that $k$ is minimal (that is, $S$ spans $\RR^k$).  If~$k = 3$, then $S$ is factorial.  If~$k = 2$, then $S$ has a unique Betti element~\cite{uniquebetti} and thus has only one nonzero catenary degree by Corollary~\ref{c:singlecatdegree}.  Lastly, if $S$ is a numerical monoid, then requiring $2 \in \mathsf C(S)$ forces the atoms of $S$ to contain an arithmetic sequence of length~3.  Since $S$ has only 3 atoms, \cite[Theorem~3.1]{catenaryarith} implies $\mathsf C(S) = \{0, 2, \mathsf c(S)\}$.  
\end{remark}





\end{document}